\documentclass[12pt]{article}
\usepackage{hyperref}
\usepackage{cite}
\usepackage[utf8]{inputenc}
\usepackage[english]{babel}
\usepackage{amsmath}
\usepackage{amsfonts}
\usepackage{amssymb}
\usepackage{amsthm}
\usepackage{xy}
\author{Jaime Castro P\'erez\footnote{jcastrop@itesm.mx,\;corresponding\;author} \\ \textit{Escuela de Ingenier\'ia y Ciencias, Instituto Tecnol\'ologico y de} \\ \textit{Estudios Superiores de Monterrey} \\ \textit{Calle del Puente 222, Tlalpan, 14380, M\'exico D.F., M\'exico.} \\ Mauricio Medina B\'arcenas\footnote{mmedina@matem.unam.mx} \; Jos\'e R\'ios Montes\footnote{jrios@matem.unam.mx} \\  Angel Zald\'ivar\footnote{zaldivar@matem.unam.mx}\\ \textit{Instituto de Matem\'aticas, Universidad Nacional} \\ \textit{Aut\'onoma de M\'exico} \\ \textit{Area de la Investigaci\'on Cient\'ifica, Circuito Exterior, C.U.,} \\ \textit{04510, M\'exico D.F., M\'exico.}}
\title{ON THE STRUCTURE OF GOLDIE MODULES}

\theoremstyle{plain}
\newtheorem{thm}{\protect\theoremname}[section]
  \theoremstyle{definition}
  \newtheorem{rem}[thm]{\protect\remarkname}
 \theoremstyle{definition}
  \newtheorem{example}[thm]{\protect\examplename}
\theoremstyle{definition}
  
  \theoremstyle{plain}
  \newtheorem{cor}[thm]{\protect\corollaryname}
  \theoremstyle{plain}
  \newtheorem{lemma}[thm]{\protect\lemmaname}
  \theoremstyle{plain}
  \newtheorem{prop}[thm]{\protect\propositionname}
    \theoremstyle{definition}
  \newtheorem{dfn}[thm]{\protect\definitionname}

\makeatother

\providecommand{\lemmaname}{Lemma}
\providecommand{\corollaryname}{Corollary}
\providecommand{\propositionname}{Proposition}
\providecommand{\definitionname}{Definition}
\providecommand{\theoremname}{Theorem}
\providecommand{\remarkname}{Remark}
\providecommand{\examplename}{Example}

\newcommand{\E}{E^{[M]}}
\newcommand{\sm}{\sigma[M]}
\newcommand{\Hom}{Hom_R}

\begin{document}
\maketitle

\begin{abstract}
Given a semiprime Goldie module $M$ projective in $\sm$ we study decompositions on its $M$-injective hull $\widehat{M}$ in terms of the minimal prime in $M$ submodules. With this, we characterize the semiprime Goldie modules in $\mathbb{Z}$-Mod and make a decomposition of the endomorphism ring of $\widehat{M}$. Also, we investigate the relations among semiprime Goldie modules, $QI$-modules and co-semisimple modules extending results on left $QI$-rings and $V$-rings.

\end{abstract}

\section{Introduction and preliminars}

In this paper we study in a deeper way the Goldie modules structure. Goldie Modules where introduce in \cite{maugoldie} as a generalization of left Goldie ring. A Goldie module is such that it satisfies ACC on left annihilators and has finite uniform dimension. In \cite{mauacc} are studied, in general modules which satisfy ACC on left annihilators; in this article the authors show that a module $M$ projective in $\sm$ and semiprime which satisfies ACC on left annihilators has finitely many minimal prime submodules \cite[Theorem 2.2]{mauacc}. Using this fact, if $M$ is a semiprime Goldie module projective in  $\sm$ we study in this manuscript the relation between the $M$-injective hull of $M$ and the $M$-injective hulls of the factor modules given by the minimal prime submodules of $M$. We prove in Corollary \ref{4}.\textit{1} that the $M$-injective hull is isomorphic to the direct sum of the $M$-injective hulls of the factor modules given by the minimal prime submodules of $M$. Moreover the summands of this decomposition are homological independent.

In \cite[Proposition 2.25]{maugoldie} is proved that if $M$ is a semiprime Goldie module projective in $\sm$ with $P_1,...,P_n$ its minimal prime submodules the each $M/P_i$ is a prime Goldie module. With that, we prove in Theorem \ref{10} that the endomorphism ring of the $M$-injective hull of $M$ is a direct product of simple artinian rings  and each simple artinian ring is a right quotient ring.

We investigate when a quasi-injective module in $\sm$ is injective provided $M$ is a semiprime Goldie module, this in order to prove when a semiprime Goldie module is a QI-module and generalize some results given by Faith in \cite{faithhereditary}.

Talking about QI-modules arise co-semisimple modules. We give some results on co-semisimple Goldie modules, we show that every indecomposable injective summand of the $M$-injective hull of a co-semisimple Goldie module is FI-simple Proposition \ref{1.16}, also we see that every co-semisimple prime Goldie module is FI-simple Proposition \ref{1.12}. This gives as consequence that if $M$ is a co-semisimple module then $M$ has finitely many prime submodules and every factor $M/N$ with $N$ a fully invariant submodule of $M$ is a semiprime Goldie module.

We organized the content of this paper as follows: Fist section is this introduction and the necessary preliminars for the develop of this theory.

Second section concerns to give some results on the structure of semiprime Goldie modules, we do this showing some decompositions in their $M$-injective hulls. Also, we show examples of semiprime Goldie modules and characterize all of them in $\mathbb{Z}$-Mod.

In section three we study quasi-injective modules in $\sm$ and $QI$-modules. We give conditions on a semiprime Goldie module in order to every quasi-injective module to be injective in $\sm$. Also at the end of this section we give some observations on hereditary semirime Goldie modules. 

The last section describe co-semisimple Goldie modules, here is shown that a co-semisimple Goldie module is isomorphic to a direct product of prime Goldie $FI$-simple modules.

Through this paper $R$ will denote an associative ring with unitary element and all $R$-modules are left unital modules. $R$-Mod denotes the category of left $R$-modules. For a submodule $N$ of a module $M$ we write $N\leq M$ if $N$ is a proper submodule we just write $N<M$ and if it is essential $N\leq_eM$. A fully invariant submodule $N$ of $M$, denoted $N\leq_{FI}M$ is a such that $f(N)\leq N$ for all $f\in End_R(M)$. When a module has no nonzero fully invariant proper submodules it is called FI-simple.

Remember that if $_RM$ is an $R$-module, $\sm$ denotes the full subcategory of $R$-Mod consisting of those modules $M$-subgenerated (see \cite{wisbauerfoundations}). To give a better develop of the content, for $X$ an $R$-module and $K\in\sigma[X]$ we will denote by $E^{[X]}(K)$ the $X$-injective hull of $K$. When $X=K$ we just write $E^{[X]}(X)=\widehat{X}$.

A module $X\in\sm$ is called $M$-singular if there exists a exact sequence
\[0\to K\to L\to X\to 0\]
such that $K\leq_eL$. Every module $N$ in $\sm$ has a largest $M$-singular submodule $\mathcal{Z}(N)$, if $\mathcal{Z}(N)=0$ then we say that $N$ is non $M$-singular.

Given two submodules $N,K\leq M$, in \cite{BicanPr} is defined the product of $N$ and $K$ in $M$ as 
\[N_MK=\sum\{f(N)|f\in\Hom(M,K)\}\]

In \cite[Proposition 1.3]{PepeGab} can be found some proprieties of this product. In general, this product is not associative but if $M$ is projective in $\sm$ it is \cite[Proposition 5.6]{beachy2002m}. Moreover

\begin{lemma}
Let $M$ be projective in $\sm$. Let $K\leq M$ and $\{X_i\}_I$ a family of modules in $\sm$. Then 
\[K_M\left(\sum_IX_i\right)=\sum\left(K_M{X_i}\right)\]
\end{lemma}

\begin{proof}
By \cite[Proposition 1.3]{PepeGab} we have that $\sum\left(K_M{X_i}\right)\leq K_M\left(\sum_IX_i\right)$. 

Now, let $f:M\to \sum_IX_i$. We have a canonical epimorphism $g:\bigoplus_IX_i\to\sum_IX_i$. Since $M$ is projective in $\sm$ there exist $\{f_i:M\to X_i\}_I$ such that $g\circ\bigoplus_If_i=f$. Hence 
\[f(K)=(g\circ\bigoplus_If_i)(K)=\sum f_i(K)\leq\sum(K_M{X_i})\]
Thus $K_M\left(\sum_IX_i\right)\leq\sum\left(K_M{X_i}\right)$.
\end{proof}
 
With this product, in \cite{raggiprime} and \cite{raggisemiprime} the authors define prime and semiprime submodule respectively in the following way

\begin{dfn}
Let $M$ be an $R$-module and $P\leq_{FI}M$. It is said $P$ is a prime submodule in $M$ if whenever $N_MK\leq P$ with $N$ and $K$ fully invariant submodules of $M$ then $N\leq P$ or $K\leq P$. If $0$ is prime in $M$ we say $M$ is a prime module.
\end{dfn}

\begin{dfn}
Let $M$ be an $R$-module and $N\leq_{FI}M$. It is said $N$ is a semiprime submodule in $M$ if whenever $K_MK\leq P$ with $K$ a fully invariant submodules of $M$ then $K\leq P$. If $0$ is semiprime in $M$ we say $M$ is a semiprime module.
\end{dfn}

Given $K,X\in\sm$, the annihilator of $K$ in $X$ is defined as
\[Ann_X(K)=\bigcap\{Ker(f)|f\in\Hom(X,K)\]
This submodule has the propriety that is the greatest submodule of $X$ such that $Ann_M(K)_XK=0$. If we consider $K\leq M$ with $M$ projective in $\sm$ then can be defined the right annihilator of $K$ as
\[Ann^r_M(K)=\sum\{L\leq M|K_ML=0\}\]
In \cite[Remark 1.15]{mauacc} is proved that $Ann^r_M(K)$ is a fully invariant submodule of $M$ and is the greatest submodule of $M$ such that $K_MAnn^r_M(K)=0$. One advantage of working with semiprime modules is that $Ann_M(K)=Ann^r_M(K)$ for all $K\leq M$ (\cite[Proposition 1.16]{mauacc}).

Following \cite{PepeKrull}

\begin{dfn}
Let $M\in{R-Mod}$. We call a left annihilator in $M$ a submodule 
\[\mathcal{A}_X=\bigcap\{Ker(f)|f\in{X}\}\]
for some $X\subseteq{End_R(M)}$.
\end{dfn}

As we mentioned before, in \cite[Theorem 2.2]{mauacc} was proved that a semiprime module $M$ projective in $\sm$ and which satisfies ACC on left annihilators has finitely many minimal prime submodules; call these $P_1,...,P_n$. The minimal primes gives a decomposition of $\widehat{M}$ in the following way
\[\widehat{M}=\E(N_1)\oplus\cdots\oplus\E(N_n)\]
where $N_i=Ann_M(P_i)$. 

In the next section we present a decomposition of $\widehat{M}$ equivalent to the last one in the sense of Krull-Remak-Schmidt, using the $M$-injective hulls of $M/P_i$.

\section{A Decomposition of a Goldie Module}

\begin{lemma}\label{1}
Let $M$ be projective in $\sigma[M]$ and semiprime. Suppose $M$ satisfies ACC on left annihilators and $P_1,...,P_n$ are the minimal prime in $M$ submodules. If $N_i=Ann_M(P_i)$ for each $1\leq{i}\leq{n}$ then $N_i$ is a pseudocomplement of $\underset{j\neq{i}}{\bigoplus}{N_j}$. 
\end{lemma}

\begin{proof}
It is enough to prove that $N_1$ is a p.c. of $\bigoplus_{i=2}^n{N_i}$. Let $L\leq{M}$ such that $L\cap{(\bigoplus_{i=2}^n{N_i})}=0$ and $N_1\leq{L}$. Since $\bigoplus_{i=2}^n{N_i}$ is a fully invariant submodule then $(\bigoplus_{i=2}^n{N_i})_ML\leq{(\bigoplus_{i=2}^n{N_i})\cap{L}}=0$. Since ${N_i}_M{L}\leq{N_i}$ then $\bigoplus_{i=2}^n{({N_i}_M{L})}=0$. Hence ${N_i}_M{L}=0$ for all $2\leq{i}\leq{n}$. So, by \cite[Theorem 2.2]{mauacc} and \cite[Proposition 1.22]{maugoldie}, $L\leq{Ann_M(N_i)}=Ann_M(Ann_M(P_i))=P_i$ for all $2\leq{i}\leq{n}$. Then $L\leq\bigcap_{i=2}^n{P_i}=Ann_M(P_1)=N_1$. Thus $L=N_1$.
\end{proof}

\begin{cor}\label{01}
Let $M$ be projective in $\sigma[M]$ and semiprime. Suppose $M$ satisfies ACC on left annihilators and $P_1,...,P_n$ are the minimal prime in $M$ submodules. If $N_i=Ann_M(P_i)$ for each $1\leq{i}\leq{n}$ then $\widehat{M}=\E(N_1)\oplus\cdots\E(N_n)$.
\end{cor}

\begin{proof}
By Lemma \ref{01} $\bigoplus_{i=1}^nN_i\leq_eM$. Thus we have done.
\end{proof}

\begin{lemma}\label{02}
Let $M$ be projective in $\sigma[M]$ and semiprime. Suppose that $M$ satisfies ACC on left annihilators and $P_1,...,P_n$ are the minimal prime in $M$ submodules. If $N_i=Ann_M(P_i)$ for each $1\leq{i}\leq{n}$ then $P_i$ is a pseudocomplement of $N_i$ which contains $\underset{j\neq{i}}{\bigoplus}{N_j}$ for all $1\leq{i}\leq{n}$. Moreover $\underset{j\neq{i}}{\bigoplus}{N_j}\leq_{e}P_i$ for all $1\leq{i}\leq{n}$.
\end{lemma}

\begin{proof}
Fix $1\leq{i}\leq{n}$. Let $L\leq{M}$ such that $P_i\leq{L}$ and $L\cap{N_i}=0$. Since $N_i$ is a fully invariant submodule of $M$ then ${N_i}_ML\leq{N_i\cap{L}}=0$. So $L\leq{Ann_M(N_i)}=Ann_M(Ann_M(P_i))=P_i$. Thus $L=P_i$. 

Now by \cite[Lemma 2.25]{mauacc} we have that $N_{i}=Ann_{M}\left( P_{i}\right) =\underset{i\neq j}{\cap }P_{j}$. Hence $N_{j}\subseteq P_{i}$ for all $j\neq i$. Then $\underset{j\neq{i}}{\bigoplus}{N_j}\leq{P_i}$ and by Lemma \ref{01}, $\underset{j\neq{i}}{\bigoplus}{N_j}\leq_{e}P_i$.
\end{proof}

\begin{lemma}\label{2}
Let $M$ be projective in $\sigma[M]$ and semiprime. Suppose that $M$ satisfies ACC on left annihilators and $P_{1},P_{2},...,P_{n}$ are the minimal prime in $M$ submodules. If $N_{i}=Ann_{M}( P_{i})$ for $1\leq i\leq n$, then $P_{i}+N_{i}\subseteq _{e}M$ for all $1\leq i\leq n$.
\end{lemma}

\begin{proof}
By Lemmas \ref{01} and \ref{02}, $N_i$ and $P_i$ are psudocomplements one each other for all $1\leq{i}\leq{n}$, thus $N_i+P_i=N_i\oplus{P_i}\leq_{e}{M}$ for all $1\leq{i}\leq{n}$.
\end{proof}

\begin{lemma}\label{21}
Let $M$ be projective in $\sigma[M]$ and semiprime. Suppose that $M$ satisfies ACC on left annihilators and $P_{1},P_{2},...,P_{n}$ are the minimal prime in $M$ submodules. If $N_{i}=Ann_{M}\left( P_{i}\right) $ for $1\leq i\leq n$, then 
\[P_i=M\cap\bigoplus_{j\neq{i}}{\E(N_j)}\]
for all $1\leq{i}\leq{n}$.
\end{lemma}

\begin{proof}
By Corollary \ref{01}, $\widehat{M}=\E(N_1)\oplus...\oplus{\E(N_n)}$. By Lemma \ref{2}, $\widehat{M}=\E(P_i)\oplus{\E(N_i)}$ for each $1\leq{i}\leq{n}$, moreover $\E(P_i)=\bigoplus_{j\neq{i}}{\E(N_j)}$ by Lemma \ref{02}. 

We have that $N_i\cap(M\cap{\E(P_i)})=N_i\cap{\E(P_i)}=0$, since $P_i$ is p.c. of $N_i$ then $P_i=M\cap{\E(P_i)}$. Thus $P_i=M\cap\underset{j\neq{i}}{\bigoplus}{\E(N_j)}$ for all $1\leq{i}\leq{n}$.
\end{proof}

\begin{prop}\label{3}
Let $M$ be projective in $\sigma[M]$ and semiprime. Suppose that $M$ satisfies ACC on left annihilators and $P_{1},P_{2},...,P_{n}$ are the minimal prime in $M$ submodules. If $N_{i}=Ann_{M}\left( P_{i}\right) $ for $1\leq i\leq n$, then the
morphism 
\[\Psi :M\rightarrow M/P_{1}\oplus M/P_{2}\oplus ...\oplus M/P_{n}\]
given by $\Psi \left( m\right) =\left( m+P_{1},m+P_{2,}....,m+P_{n}\right) $ is monomorphism and $Im\Psi \subseteq_e\bigoplus\limits_{i=1}^{n}M/P_{i}$.
\end{prop}

\begin{proof}
By \cite[Corollary 1.15]{maugoldie} we have that, $\overset{n}{\underset{i=1}{\cap }}P_{i}=0$. Thus $\Psi $ is monomorphism. Now let $0\neq \left(m_{1}+P_{1},m_{2}+P_{2,}....,m_{n}+P_{n}\right) \in\bigoplus\limits_{i=1}^{n}M/P_{i}$. \ By \ref{2} $\bigoplus\limits_{i=1}^{n}\left( P_{i}+\underset{i\neq j}{\cap }P_{j}\right) \subseteq _{e}M^{n}$. So there exists $r\in R$ such that 
\[0\neq{r}\left( m_{1},m_{2}{}_{,}....,m_{n}\right) \in\bigoplus\limits_{i=1}^{n}\left( P_{i}+\underset{i\neq j}{\cap }P_{j}\right)\]
Hence $rm_{i}\in P_{i}+\underset{i\neq j}{\cap }P_{j}$ for $1\leq i\leq n$. Thus there exists $x_{i}\in P_{i}$ and $y_{i}\in \underset{i\neq j}{\cap }P_{j}$ such that $rm_{i}=x_{i}+y_{i}$ for every $1\leq i\leq n$. We claim that $r\left( m_{1}+P_{1},m_{2}+P_{2,}....,m_{n}+P_{n}\right) \in Im\Psi $. In fact let $m=y_{1}+y_{2}+...+y_{n}$, then 
\[\Psi \left( m\right)=\left( y_{1}+y_{2}+...+y_{n}+P_{1},....,y_{1}+y_{2}+...+y_{n}+P_{n}\right)\]
\[=\left( y_{1}+P_{1},....,y_{n}+P_{n}\right) =\left(rm_{1}-x_{1}+P_{1},....,rm_{n}-x_{n}+P_{n}\right)\]
\[=r\left( m_{1}+P_{1},....,m_{n}+P_{n}\right)\].
\end{proof}

\begin{cor}\label{4}
Let $M$ be projective in $\sigma[M]$ and semiprime. Suppose that $M$ satisfies ACC on left annihilators and $P_{1},P_{2},...,P_{n}$ are the minimal prime in $M$ submodules, then
\begin{enumerate}
	\item $\widehat{M}\cong{\E(M/P_{1})\oplus{\E(M/P_{2})}\oplus...\oplus{\E(M/P_{n})}}$.
	\item If in addition $M$ is a Goldie module, $M/P_{i}$ has finite uniform dimension for all $0\leq i\leq n$.
\end{enumerate}
\end{cor}

\begin{proof}
\textit{1}. By Proposition \ref{3} $\Psi $ is a monomorphism and $Im\Psi \subseteq _{e}\bigoplus\limits_{i=1}^{n}M/P_{i}$. Thus
we have the result.

\textit{2}. Since $M$ is Goldie module, then $M$ has finite uniform dimension By (\textit{1})
we have that $M/P_{i}$ has finite uniform dimension for all $0\leq i\leq n$.
\end{proof}

The next corollary says that the hypothesis every quotient $M/P_i$ has finite uniform dimension of \cite[Proposition 2.25]{maugoldie} can be removed.

\begin{cor}\label{5}
Let $M$ be projective in $\sigma[M]$ and semiprime. Suppose that $M $ has a finitely many minimal prime submodules $P_{1},P_{2},...,P_{n}$. Then $M$ is a Goldie module if and only if each $M/P_{i}$ is a Goldie module.
\end{cor}

\begin{proof}
Note that by Corollary \ref{4}.\textit{2} if $M$ is a Goldie module then each quotient $M/P_i$ has finite uniform dimension, so it follows by \cite[Proposition 2.25]{maugoldie}
\end{proof}


\begin{lemma}\label{2.1}
Let $M$ be projective in $\sm$. If $M$ is a semiprime Goldie module then $M^n$ for all $n>0$.
\end{lemma}

\begin{proof}
We have that $\sm=\sigma[M^n]$ then $M^n$ is projective in $\sigma[M^n]$. Since $M$ has finite uniform dimension then $M^n$ so does. By \cite[Theorem 2.8]{maugoldie} $M$ is essentially compressible then $M^n$ is essentially compressible by \cite[Proposition 1.2]{smithessentially}. Thus by \cite[Lemma 2.7 and Theorem 2.8]{maugoldie} $M^n$ is a semiprime Goldie module.
\end{proof}

\begin{cor}\label{2.2}
Let $R$ be a ring. If $R$ is a left Goldie ring then $M_n(R)$ is a right Goldie ring for all $n>0$.
\end{cor}

\begin{proof}
By Proposition \ref{2.1} $R^n$ is a semiprime Goldie module for all $n>0$. By \cite[Theorem 2.22]{maugoldie} $End_R(R^n)\cong M_n(R)$ is a semiprime right Goldie ring.
\end{proof}

\begin{prop}\label{2.3}
Let $M$ be projective in $\sm$. Suppose $M$ has nonzero soclo. If $M$ is a prime module and satisfies ACC on left annihilators then $M$ is semisimple artinian and FI-simple.
\end{prop}

\begin{proof}
Let $0\neq{m}\in M$. Since $Zoc(M)\neq 0$ and $M$ is a prime module then $0\neq Zoc(M)_MRm\subseteq Zoc(M)\cap Rm$. Hence $Zoc(M)\leq_e M$. Let $U_1$ be a minimal submodule of $M$. By \cite[Corollary 1.17]{maugoldie} $U_1$ is a direct summand, so $M=U_1\oplus V_1$. If $0\neq V_1$, since $Zoc(M)\leq_e M$ then there exists a minimal submodule $U_2\leq V_1$. Then $M=U_1\oplus U_2\oplus V_2$. If $V_2\neq 0$ there exists a minimal submodule $U_3\leq V_3$ such that $M=U_1\oplus U_2\oplus U_3\oplus V_3$. Following in this way, we get an ascending chain 
\[U_1\leq U_1\oplus U_2\leq U_1\oplus U_2\oplus U_3\leq...\]
Notice that $U_1\oplus...\oplus U_i=Ker f$ where $f$ is the endomorphism of $M$ given by 
\[M\to M/(U_1\oplus...\oplus U_i)\cong V_i\hookrightarrow{M}\]
Thus, last chain is an ascending chain of annihilators, so it must stop in a finite step. Then $M$ is semisimple artinian.

Suppose that $M=U_1\oplus...\oplus U_n$. If $\Hom(U_i,U_j)=0$ then ${U_i}_MU_j=0$, but $M$ is prime. Thus $\Hom(U_i,U_j)\neq 0$, so $U_i\cong U_j$ $1\leq i,j\leq n$.
\end{proof}

\begin{cor}\label{2.4}
Let $M$ be projective in $\sm$. Suppose $M$ has essential soclo. If $M$ is a semiprime  Goldie module then $M$ is semisple artinian.
\end{cor}

\begin{proof}
By Proposition \ref{3} there exists an essential monomorphism $\Psi:M\to M/P_1\oplus\cdots\oplus M/P_n$ where $P_i$ are the minimal prime in $M$ submodules. Hence each $M/P_i$ has nonzero socle and by Corollary \ref{5} $M/P_i$ is a Goldie module. Thus $M/P_i$ is semisimple artinian and $FI$-simple for all $1\leq i\leq n$ by Proposition \ref{2.3}. Then $M\cong M/P_1\oplus\cdots\oplus M/P_n$ and hence semisimple artinian.
\end{proof}

\begin{example}\label{2.5}
In $\mathbb{Z}-Mod$, a module $M$ projective in $\sm$ is a semiprime Goldie module if and only if $M$ is semisimple artinian or $M$ is free of finite rank. 
\end{example}

\begin{proof}
It is clear that a semisimple artinian module $M$ is projective in $\sm$ and a semiprime Goldie module and by Proposition \ref{2.1} every free module of finite rank is a semiprime Goldie module. Now, suppose that $_\mathbb{Z}M$ is a semiprime Goldie module and projective in $\sm$.  Recall that in $\mathbb{Z}$-Mod an indecomposable injective module is isomorphic to $\mathbb{Q}$ or $\mathbb{Z}_{p^\infty}$ for some prime $p$. 

Let $U$ be an uniform submodule of $M$. By definition $\E(U)=tr^M(E^{[\mathbb{Z}]}(U))$. Suppose $E^{[\mathbb{Z}]}(U)\cong\mathbb{Q}$, then $\E(U)\leq\mathbb{Q}$. Since $\mathbb{Q}$ is $FI$-simple then $\E(U)=\mathbb{Q}$. This implies that $\mathbb{Q}\hookrightarrow\widehat{M}\in\sm$. Hence $\sm=\mathbb{Z}-Mod$. Since $M$ is projective in $\sm=\mathbb{Z}-Mod$ then $M$ is a free module and since $M$ has finite uniform dimension then it has finite rank.

Let $\bigoplus_{i=1}^nU_i\leq_eM$ with $U_i$ uniform. If one $U_i$ is a torsion free group then $M$ is free because above. So, we can suppose that every $U_i$ is a torsion group. Then, $M$ has essential soclo. By Corollary \ref{2.4} $M$ is semisimple artinian. 
\end{proof}


\begin{prop}\label{51}
Let $M$ be projective in $\sigma[M]$ and semiprime. Suppose that $M$ is a Goldie Module
and $P_{1},P_{2},...,P_{n}$ are the minimal prime in $M$ submodules, then $\E(N_i)\cong{\E(M/P_i)}$ where $N_i=Ann_M(P_i)$. Moreover, $M/P_i$ contains an essential submodule isomorphic to $N_i$. 
\end{prop}

\begin{proof}
By \cite[Proposition 4.5]{PepeGab}, $Ass_M(M/P_i)=\{P_i\}$. By Corollary \ref{4}.(\textit{1}) $\widehat{M}\cong{\E(M/P_{1})\oplus{\E(M/P_{2})}\oplus...\oplus{\E(M/P_{n})}}$. Then, following the proof of \cite[Theorem 2.20]{mauacc} there exist indecomposable injective modules $F_1$,..,$F_n$ in $\sigma[M]$ and $l_1,...,l_n\in\mathbb{N}$ such that $\widehat{M}\cong{F_1^{l_1}\oplus...\oplus{F_n^{l_n}}}$ and each $\E(M/P_i)\cong{F_i^{l_i}}$. By \cite[Theorem 2.20]{mauacc} $\widehat{M}\cong{E_1^{k_1}\oplus...\oplus{E_n^{k_n}}}$ with $\E(N_i)\cong{E_i^{k_i}}$ and $k_1,...,k_n\in\mathbb{N}$. Then by Krull-Remak-Schmidt-Azumaya Theorem, $E_i\cong{F_i}$ and $l_i=k_i$. Thus $\E(M/P_i)\cong{\E(N_i)}$.

Now, by Lemma \ref{2}, $N_i\oplus{P_i}\subseteq_{e}{M}$. Then there is an essential monomorphism $N_i\hookrightarrow{M/P_i}$. In fact, $N_i$ and $P_i$ are pseudocomplements one of each other.

\end{proof}

\begin{lemma}\label{12}
Let $M$ be projective in $\sigma[M]$ and semiprime. Suppose that $M$ is a Goldie Module and $P_{1},P_{2},...,P_{n}$ are the minimal prime in $M$ submodules. If $N_{i}=Ann_{M}\left( P_{i}\right) $, then $Ann_{M}\left(\E(N_{i})\right) =P_{i}$ for all $1\leq i\leq n$.
\end{lemma}

\begin{proof}
Since $N_i$ has finite uniform dimension 
\[\E(N_i)=\E(U_{i_1})\oplus...\oplus{\E(U_{i_{k_i}})}\]
where $U_{i_j}$ is an uniform submodule of $N_i$. By \cite[Lemma 2.16]{mauacc}, $Ann_M(U_{i_j})=P_i$ for all $1\leq{j}\leq{k_i}$. Now by \cite[Proposition 1.11 and Lemma 1.13]{mauacc}, $Ann_M(U_{i_j})=Ann_M(\E(U_{i_j}))$ for all $1\leq{j}\leq{k_i}$. Finally, by \cite[Proposition 1.3]{PepeGab}, $Ann_M(\E(U_{i_1})\oplus...\oplus{\E(U_{i_{k_i}})})=\bigcap_{j=1}^{k_i}{Ann_M(\E(U_{i_j}))}=P_i$. Thus $Ann_{M}\left( \E(N_{i})\right) =P_{i}$ for all $1\leq i\leq n$.
\end{proof}

\begin{prop}\label{8}
Let $M$ be projective in $\sigma[M]$ and semiprime. Suppose that $M$ is a Goldie Module
and $P_{1},P_{2},...,P_{n}$ are the minimal prime in $M$ submodules, then
\begin{enumerate}
	\item $Hom_R(\E(M/P_i),\E(M/P_j))=0$ if $i\neq{j}$.
	\item If $M$ is also a generator in $\sigma[M]$ then $\E(M/P_i)=\widehat{M/P_i}$.
\end{enumerate}
\end{prop}

\begin{proof}
\textit{1}. It follows by \cite[Proposition 2.21]{mauacc} and Proposition \ref{51}.

\textit{2}. Since $\sigma[M/P_i]\subseteq\sigma[M]$ for every $1\leq i\leq n$, then $\widehat{M/P_i}\leq \E(M/P_i)$. It is enough to show that $\E(M/P_i)\in\sigma[M/P_i]$. 

We claim that ${P_i}_M\E(M/P_i)=0$. If $N_i=Ann_M(P_i)$, by Lemma \ref{12} $Ann_M(\E(N_i))=P_i$. Now, by Proposition \ref{51} $\E(N_i)=\E(M/P_i)$. Thus ${P_i}_M\E(M/P_i)=0$.

By \cite[Proposition 1.5]{PepeFbn} $\E(M/P_i)\in\sigma[M/P_i]$.
\end{proof}

Notice that, (\textit{2}) of last proposition is false if $M$ is not a semiprime module.

\begin{example}
Let $M=\mathbb{Z}_4$, then $M$ is a Goldie module but it is not semiprime. $2\mathbb{Z}_4$ is the only minimal prime in $M$ and $\mathbb{Z}_4/2\mathbb{Z}_4\cong\mathbb{Z}_2$. We have that $E^{[\mathbb{Z}_4]}(\mathbb{Z}_4/2\mathbb{Z}_4)=\mathbb{Z}_4$, but $E^{[\mathbb{Z}_2]}(\mathbb{Z}_4/2\mathbb{Z}_4)=\mathbb{Z}_2$.
\end{example}

\begin{rem}\label{9}
Note that if $M$ is a Goldie Module then, by
Corollary \ref{5}, we have that $M/P_{i}$ is a Goldie module for every minimal prime $P_i$ submodule. If $S_{i}=End_{R}\left( M/P_{i}\right) $ and $T_{i}=End_{R}(\widehat{M/P_i}) $, then by \cite[Corollary 2.23]{maugoldie}, $T_{i}$ is a simple right artinian ring and is the classical right ring of quotients of $S_{i}$. On the  other hand, by Proposition \ref{8}, $\E(M/P_i)=\widehat{M/P_i}$. So $End_{R}(\E(M/P_i)) \cong End_{R}\left(\widehat{M/P_{i}}\right)$. Thus $End_{R}(\E(M/P_i)) $ is a simple right artinian ring and is the classical right ring of quotients of $S_{i}$. Moreover by Proposition \ref{51} and Proposition \ref{8} we have that $End_{R}(\E(M/P_i))\cong{End_{R}(\E(N_{i}))}\cong End_R(\widehat{M/P_i})$. 
\end{rem}

\begin{thm}\label{10}
Let $M$ be progenerator in $\sigma[M]$ and semiprime. Suppose that $M$ is a Goldie Module and $P_{1},P_{2},...,P_{n}$ are the minimal prime in $M$ submodules, then 
\[End_R(\widehat{M})\cong{End_R(\E(M/P_1))\times\cdots\times{End_R(\E(M/P_n))}}\]
where $End_R(\E(M/P_i))$ is a simple right artinian ring and is the classical ring of quotients of $End_R(M/P_i)$.
\end{thm}

\begin{proof}
By Corollary \ref{4}, 
\[\widehat{M}\cong{\E(M/P_{1})\oplus{\E(M/P_{2})}\oplus...\oplus{\E(M/P_{n})}}\]
Then 
\[End_R(\widehat{M})\cong{End_R(\E(M/P_{1})\oplus{\E(M/P_{2})}\oplus...\oplus{\E(M/P_{n})})}\]
By Proposition \ref{8}, $Hom_R(\E(M/P_i),\E(M/P_j))=0$ if $i\neq{j}$ hence
\[End_R(\widehat{M})\cong{End_R(\E(M/P_1))\times\cdots\times{End_R(\E(M/P_n))}}\]
The rest follows by Remark \ref{9}.
\end{proof}

\begin{cor}\label{11}
Let $M$ be progenerator in $\sigma[M]$ and semiprime. Suppose $M$ is a Goldie Module and let $T=End_R(\widehat{M})$. Then there is a bijective correspondence between the set of minimal primes in $M$ and a complete set of isomorphism classes of simple $T$-Modules.
\end{cor}

\begin{proof}
By Theorem \ref{10}
\[End_R(\widehat{M})\cong{End_R(\E(M/P_1))\times\cdots\times{End_R(\E(M/P_n))}}\]
Each $End_R(\E(M/P_i))$ is an Homogeneous component of $T$. If $H$ is a simple $T$-module, then $H\cong{T/L}$ where $L$ is a maximal right ideal of $T$. Then $H\hookrightarrow{End_R(\E(M/P_i))}$ for some $1\leq{i}\leq{n}$.
\end{proof}

\section{Goldie Modules and $M$-injective Modules}

\begin{prop}\label{1.1}
Let $M$ be projective in $\sigma[M]$ and semiprime. Suppose $M$ is a Goldie module with minimal primes submodules $P_1,...,P_n$ and $N_i=Ann_M(P_i)$ for $1\leq i\leq n$. If $Q\leq M$ is $M$-injective and a fully invariant submodule of $\widehat{M}$ then $Q=\bigoplus_{j\in J}\E(N_j)$ for some $J\subseteq\{1,...,n\}$.
\end{prop}

\begin{proof}
Following the proof of \cite[Theorem 2.20]{mauacc} there exist $E_1,...,E_n$ indecomposable injective modules in $\sigma[M]$ such that $\E(N_i)\cong E_i^{k_i}$ and $\widehat{M}=E_1^{k_1}\oplus\cdots\oplus E_n^{k_n}$. Since $Q$ is injective and has finite uniform dimension, by Krull-Remak-Schmidt-Azumaya $Q=E_1^{l_1}\oplus\cdots\oplus E_n^{l_n}$ with $0\leq l_i\leq k_i$, where if $l_i=0$ we define $E_i^{0}=0$. If $1\leq l_i< k_i$ then $Q$ is not fully invariant in $\widehat{M}$. So if $l_i\neq 0$ then $l_i=k_i$. Thus $Q=\bigoplus_{j\in J}\E(N_j)$ for some $J\subseteq\{1,...,n\}$.
\end{proof}

\begin{prop}\label{1.2}
Let $M$ be projective in $\sigma[M]$, non $M$-singular and with finite uniform dimension. If $A\in\sigma[M]$ is non $M$-singular and quasi-injective with $Ann_M(A)=0$ then $A$ is $M$-injective.
\end{prop}

\begin{proof}
Note that if $X\subseteq Hom_R(M,A)$ then $\bigcap_{f\in X}Ker(f)$ is closed in $M$ because $A$ is non $M$-singular. Consider $\Gamma=\{N\leq M|N=\bigcap_{f\in X}ker(f)\;X\subseteq Hom_R(M,A)\}$. Since $M$ has finite uniform dimension, by \cite[Proposition 6.30]{lamlectures} $M$ satisfies DCC on closed submodules. So $\Gamma$ has minimal elements. 

We have that $0=Ann_M(A)\in\Gamma$, hence there exist $f_1,...,f_n\in Hom_R(M,A)$ such that $0=\bigcap_{i=1}^nKer(f_i)$. Thus there exists a monomorphism $\varphi:M\to A^n$. Since $A^n$ is quasi-injective and $\varphi$ is a monomorphism then $A$ is $M$-injective.
\end{proof}

\begin{prop}\label{1.3}
Let $M$ be projective in $\sigma[M]$ and non $M$-singular. If $L\in \sigma[M]$ is non $M$-singular then $Ann_{\widehat{M}}(L)$ is $M$-injective.
\end{prop}

\begin{proof}
Let $K=Ann_{\widehat{M}}(L)$. Note that $K$ is closed, so it is $M$-injective.
\end{proof}

\begin{prop}\label{1.4}
Let $M$ be projective in $\sigma[M]$ and semiprime. Suppose $M$ is a Goldie module and $P_1,...,P_n$ are its minimal primes. If $N=\bigcap_{j\in J}P_j$ with $\emptyset\neq J\subseteq\{1,...,n\}$ then $M/N$ is a semiprime Goldie module.
\end{prop}

\begin{proof}
It is clear $M/N$ is a semiprime module. Since $N$ is fully invariant in $M$ then $M/N$ is projective in $\sigma[M/N]$. We have that $M$ is a semiprime Goldie module, so $M/P_i$ does for every $1\leq i\leq n$. 

There exists a monomorphism $M/N\hookrightarrow\bigoplus_{j\in J}M/P_j$. By the proof of \cite[Proposition 2.25]{maugoldie}, every factor module $M/P_j$ is non $M$-singular, hence $M/N$ is non $M$-singular. It follows $M/N$ is non $M/N$-singular. By Corollary \ref{4}.\textit{2} $M/P_j$ has finite uniform dimension for every $j\in J$, so $M/N$ has finite uniform dimension. Thus $M/N$ is a semiprime Goldie module by \cite[Theorem 2.8]{maugoldie}.
\end{proof}

\begin{rem}\label{1.5}
Let $N\leq M$ a fully invariant submodule of $M$ with $M$ projective in $\sigma[M]$. Let $P$ be a prime in $M$ such that $N\leq P$. By \cite[Proposition 18]{raggiprime}, $P/N\leq M/N$ is prime in $M/N$ if and only if $P$ is prime in $M$. So, $P/N$ is a minimal prime in $M/N$ if and only if $P$ is a minimal prime in $M$. 

Then, in the conditions of last Proposition, the minimal prime in $M/N$ submodules are $\{P_j/N|j\in J\}$. 
\end{rem}

\begin{cor}\label{1.6}
Let $M$ be projective in $\sigma[M]$. Suppose $M$ is a  semiprime Goldie module. Let $K$ be a fully invariant submodule of $\widehat{M}$ and assume $K$ is $M$-injective. If $N=K\cap M$ then $M/N$ is a semiprime Goldie module.  
\end{cor}

\begin{proof}
Let $P_1,...,P_n$ be the minimal prime in $M$ submodules. Let us denote $N_i=Ann_M(P_i)$ with $1\leq i\leq n$. By Proposition \ref{1.1} $K=\bigoplus_{j\in J}\E(N_j)$ for some $J\subseteq\{1,...,n\}$. Thus $N=\bigcap_{i\notin J}P_i$ by Lemma \ref{21}. By Proposition \ref{1.4} $M/N$ is a semiprime Goldie module. 
\end{proof}

\begin{prop}\label{1.7}
Let $M$ be progenerator in $\sigma[M]$. Suppose $M$ is a  semiprime Goldie module with minimal primes $P_1,...,P_n$. Let $K$ be a fully invariant submodule of $\widehat{M}$ and assume $K$ is $M$-injective. If $N=K\cap M$ then there exists $\emptyset\neq J\subseteq\{1,...,n\}$ such that the classical right quotients ring of $End_R(M/N)$ is isomorphic to 
\[\prod_{i\notin J}End_R(\E(M/P_i))\] 
\end{prop}

\begin{proof}
By Corollary \ref{1.6} there exist $\emptyset\neq J\subseteq\{1,...,n\}$ such that $N=\bigcap_{i\notin J}P_i$ and $M/N$ is a semiprime Goldie module. By Corollary \ref{4} and Remark \ref{1.5} we have that
\[\widehat{M/N}\cong \bigoplus_{i\notin J} E^{[M/N]}(M/P_i)\]
and 
\[Hom_R(E^{[M/N]}(M/P_i),E^{[M/N]}(M/P_l))=0\]
if $i\neq l$. Hence
\[End_R(\widehat{[M/N]})\cong\prod_{i\notin J}E^{[M/N]}(M/P_i)\]
Since $M/N$ is a semiprime Goldie module, by \cite[Theorem 2.22]{maugoldie}, $End_R(\widehat{M/N})$ is the classical right quotients ring of $End_R(M/N)$. 

Now, we claim that $E^{[M/N]}(M/P_i)=\E(M/P_i)$. By Proposition \ref{51} and Lemma \ref{12}, $Ann_M(\E(M/P_i))=P_i$ for all $1\leq i\leq n$. Hence $Ann_M(\bigoplus_{i\notin J}\E(M/P_i))=N$. We have a monomorphism from $M/N$ to $\bigoplus_{i\notin J}M/P_i$ so $N_M\E(M/N)=0$. Thus by \cite[Proposition 1.5]{PepeFbn} $\E(M/N)\in \sigma[M/N]$, it follows $E^{[M/N]}(M/P_i)=\E(M/P_i)$.
\end{proof}

\begin{prop}\label{1.8}
Let $M$ be progenerator in $\sigma[M]$ and suppose $M$ is a semiprime Goldie module. If $A\in\sigma[M]$ is non $M$-singular, quasi-injective and $Ann_{\widehat{M}}(\E(A))=0$ then $A$ is $M$-injective.
\end{prop}

\begin{proof}
By Proposition \ref{1.2}, it is enough to prove $Ann_M(A)=0$. Let $H=Ann_M(A)$. Since $\widehat{M}$ has finite uniform dimension, $\widehat{M}$ satisfies DCC on closed submodules. 

If $X\subseteq Hom_R(\widehat{M},\E(A))$, then $\bigcap_{f\in X}Ker(f)$ is closed in $\widehat{M}$ because $\E(A)$ is non $M$-singular. We have that $Ann_{\widehat{M}}(\E(A))=0$, so there exist $f_1,...,f_m\in Hom_R(\widehat{M},\E(A))$ such that $0=\bigcap_{i=1}^m Ker(f_i)$. For every $1\leq i\leq m$ consider the morphisms $\pi f_i:\widehat{M}\to \E(A)/A$ where $\pi:A\to \E(A)/A$ is the canonical projection. Since $\E(A)/A$ is $M$-singular, by \cite[Proposition 2.5]{maugoldie} $M\cap Ker(\pi f_i)=Ker((\pi f_i)|_M)\leq_eM$. Let $L=\bigcap_{i=1}^m(M\cap Ker(\pi f_i))$, then $L\leq_eM$. 

We have that $f_i(L)\leq A$ for all $1\leq i \leq m$, so $H_M{f_i(L)}=0$ for all $1\leq i\leq m$. We define $F:L\to A^m$ as $F(l)=(f_1(l),...,f_m(l))$. Therefore, $Ker(F)=\bigcap_{i=1}^m(M\cap Ker(f_i))\leq\bigcap_{i=1}^m Ker(f_i)=0$. Thus, $F$ is a monomorphism. Hence $H_ML=0$. Since $M$ is a semiprime module, by \cite[Lemma 1.9]{maugoldie} $L_MH=0$, but if $H\neq 0$ this is a contradiction because $L\leq_eM$ and $H$ is no $M$-singular. Thus $H=0$.
\end{proof}

\begin{thm}\label{1.9}
Let $M$ be progenerator in $\sigma[M]$ and suppose $M$ is a semiprime Goldie module. If $A\in\sigma[M]$ is non $M$-singular and quasi-injective then $A$ is $M$-injective.
\end{thm}

\begin{proof}
Suppose $P_1,...,P_n$ are the minimal prime in $M$ submodules and let $N_i=Ann_M(P_i)$. Consider $\E(A)$. By Proposition \ref{1.3}, $L:=Ann_{\widehat{M}}(\E(A))$ is $M$-injective, so $\widehat{M}=L\oplus K$. Let $N=M\cap L$. Then $N=\bigcap_{i\notin J}P_i$ for some $\emptyset\neq J\subseteq\{1,...,n\}$ and $M/N$ is a semiprime Goldie module by Corollary \ref{1.6}.

Note that $N\subseteq Ann_M(\E(A))$, hence $N_M(\E(A))=0$. By \cite[Proposition 1.5]{PepeFbn} $\E(A)\in\sigma[M/N]$. This implies that $\E(A)=E^{[M/N]}(A)$. Since
\[M/N=M/(M\cap L)\cong (M+L)/L\leq \widehat{M}/L\cong K\]
then $\sigma[M/N]\subseteq\sigma[K]$.

If $k\in K$ is such that $f(k)=0$ for all $f:K\to \E(A)$ then $k\in L\cap K=0$. Hence, $Ann_K(\E(A))=0$. By Corollary \ref{4} and Remark \ref{1.5}, 
\[\widehat{M/N}\cong \bigoplus_{i\notin J}E^{[M/N]}(M/P_i)\]
Moreover, by the proof of Proposition \ref{1.7}
\[\widehat{M/N}\cong\bigoplus_{i\notin J} \E(M/P_i)\cong\bigoplus_{i\notin J}\E(N_i)\cong K\]
Thus $\widehat{M/N}=K\in \sigma[K]=\sigma[M/N]$.

Hence, we are in the hypothesis of Proposition \ref{1.8}, thus $A$ is $M/N$-injective, which implies that $A=E^{[M/N]}(A)=\E(A)$. Therefore $A$ is $M$-injective.
\end{proof}

\begin{dfn}
Let $M$ be an $R$-module. It is said $M$ is a $QI$-module if every $A\in\sigma[M]$ quasi-injective module is $M$-injective.

A ring $R$ is left $QI$-ring if it is a $QI$-module as left module over itself.
\end{dfn}

In \cite{daunsqi} the authors give many characterizations of $QI$-modules. Here, we show some generalizations of Faith's results that appear in \cite{faithhereditary}

\begin{cor}\label{1.10}
Let $M$ be progenerator in $\sigma[M]$. Suppose $M$ is a semiprime Goldie module. If every $A\in \sigma[M]$ quasi-injective and $M$-singular module is $M$-injective then $M$ is a $QI$-module.
\end{cor}

\begin{proof}
Let $A\in\sigma[M]$ be a quasi-injective module. If $A$ is non $M$-singular then $A$ is $M$-injective by Theorem \ref{1.9}. Suppose $\mathcal{Z}(A)\neq 0$. 

We claim $\mathcal{Z}(A)$ is quasi-injective. Let $f:\E(\mathcal{Z}(A))\to \E(\mathcal{Z}(A))$. Then there exists $\overline{f}\in End_R(\E(A))$ such that $\overline{f}|_{\E(A)}=f$. Since $A$ is quasi-injective $\overline{f}(A)\leq A$, hence $\overline{f}(\mathcal{Z}(A))\leq \mathcal{Z}(A)$. Therefore $f(\mathcal{Z}(A))\leq \mathcal{Z}(A)$. Thus $\mathcal{Z}(A)$ is quasi-injective.

By hypothesis $\mathcal{Z}(A)$ is $M$-injective, so $\mathcal{Z}(A)$ is a direct summand of $A$, i.e., $A=\mathcal{Z}(A)\oplus K$. Since $A$ is quasi-injective, $K$ is quasi-injective; also $K$ is non $M$-singular. By Theorem \ref{1.9} $K$ is $M$-injective. Thus $A$ is $M$-injective. 
\end{proof}


\begin{dfn}\label{2.6}
Let $M$ be an $R$-module. Is is said $M$ is hereditary if every submodule of $M$ is projective in $\sm$.

A ring $R$ is left hereditary if it is as left $R$-module.
\end{dfn}

\begin{lemma}\label{2.6.1}
Let $M$ be projective in $\sm$ and semiprime. Suppose $M$ is non $M$-singular. If $U\leq M$ is an uniform submodule then $U$ is compressible.
\end{lemma}

\begin{proof}
Let $U\leq M$ uniform and $0\neq V\leq U$. If $U_MV=0$ then $V_MV=0$ but $M$ is semiprime then $V=0$. Contradiction. So, $U_MV\neq 0$. Hence there exists $f:M\to V$ such that $f(U)\neq 0$, since $V$ is non $M$-singular then $U\cap Ker(f)=0$. Thus there exists a monomorphis $f|_U:U\to V$.
\end{proof}

\begin{prop}\label{2.7}
Let $M$ be hereditary and a semiprime Goldie module. Then every submodule $N\leq M$ is a semiprime Goldie module.
\end{prop}

\begin{proof}
Since $M$ is hereditary $N$ is projective in $\sm$ and thus it is projective in $\sigma[N]$. Since $M$ has finite uniform dimension there exist uniform submodules $U_1,...,U_n$ such that $\bigoplus_{i=1}^nU_i\leq_e M$, aslo by Lemma \ref{2.6.1} each $U_i$ is compressible. By \cite[Theorem 2.8]{maugoldie} $M$ is a essentially compressible, so there exists a monomorphims $M\to \bigoplus_{i=1}^nU_i$. Then, by \cite[Proposition 1.8]{smithessentially} $N$ is essentially compressible. Thus, by \cite[Teorema 2.8]{maugoldie} $N$ is a semiprime Goldie module.
\end{proof}

\begin{cor}\label{2.8}
Let $R$ be a left hereditary ring. If $R$ is a semiprime left Goldie ring then every left ideal is a Goldie module.
\end{cor}

\begin{rem}
As example of a ring that satisfies the conditions of Corollary \ref{2.8} is the Cozzen's ring given in \cite{cozzenshomological}. 

Boyle's conjecture says that every left $QI$-ring is left hereditary, if this conjecture was true then a $QI$-ring would satisfy the conditions of Corollary \ref{2.8}
\end{rem}

\section{Co-semisimple Goldie Modules}

\begin{dfn}
Let $M$ be an $R$-module. $M$ is \emph{co-semisimple} if every simple module in $\sigma[M]$ is $M$-injective. 

If $M= {_RR}$, this is the definition of a \emph{left $V$-ring. }
\end{dfn}

\begin{rem}
It is well known that if $R$ is a left $QI$-ring then $R$ is left noetherian $V$-ring. In the case of modules, a $QI$-module is not, in general, a noetherian module. In fact, let $R$ be a ring such that $\{S_i\}_I$ is an infinity family of non isomorphic simple $R$-modules. Consider $M=\bigoplus_I S_i$, then $M$ is a $QI$-module but $M$ has no finite uniform dimension so $M$ is not noetherian.
\end{rem}

The following characterization of co-semisimple modules is given in \cite[23.1]{wisbauerfoundations}

\begin{prop}\label{1.11}
For an $R$-module $M$ the following statements are equivalent:
\begin{enumerate}
	\item $M$ is co-semisimple.
	\item Any proper submodule of $M$ is an intersection of maximal submodules.
\end{enumerate}
\end{prop} 

\begin{prop}
Let $M$ be projective in $\sigma[M]$. If $M$ is co-semisimple then every fully invariant submodule of $M$ is idempotent.
\end{prop}

\begin{proof}
Let $N$ be a proper fully invariant submodule of $M$. Suppose $N^2< N$. Since $M$ is co-semisimple then, by Proposition \ref{1.11} $N^2=\bigcap_I\{\mathcal{M}_i\}$ with $\mathcal{M}<M$ maximal for all $i\in I$. By \cite[Remark 4.23]{mauquantale} $N\leq \bigcap_I\{\mathcal{M}_i\}=N^2$. Thus $N=N^2$.
\end{proof}

\begin{lemma}\label{13}
Let $M\in R$-Mod and projective in $\sigma %
\left[ M\right] $. If $M$ is a prime module and $0\neq N\varsubsetneq M$ is a
fully invariant submodule, then $N\subseteq _{e}M$.
\end{lemma}

\begin{prop}\label{1.12}
Let $M$ be projective in $\sigma[M]$. If $M$ is a prime Goldie module and co-semisimple then $M$ is $FI$-simple.
\end{prop}

\begin{proof}
Let $0\neq N$ be a fully invariant submodule of $M$. By Lemma \ref{13}, $N\leq_e M$. By \cite[Theorem 2.8]{maugoldie} there exists a monomorphism $f:M\to N$. If $N$ is proper in $M$ there exists a maximal submodule $\mathcal{M}< M$ such that $N\leq \mathcal{M}$. Hence, we have a epimorphism $h:M/N\to M/\mathcal{M}\cong S$ with $S$ a simple module. By \cite[Lemma 5.4]{beachy2002m} $h(N_M(M/N))=N_MS$. Since $N_M(M/N)=0$ by \cite[Proposition 1.8]{PepeGab}, then $N_MS=0$, i.e., $g(N)=0$ for all $g\in Hom_R(M,S)$. 

Since $S$ is injective in $\sigma[M]$, there exists $g:M\to S$ such that $gf=\pi$ where $\pi:M\to M/\mathcal{M}$ is the canonical projection. Hence $\pi(M)=gf(M)\leq g(N)=0$. Contradiction. Thus $M=N$.
\end{proof}

\begin{rem}
Notice that by \cite[Remark 4.23]{mauquantale} every fully invariant submodule of $M$ is semiprime in $M$. In particular $M$ is a semiprime module. Hence if $M$ is co-semisimple and Goldie module, by \cite[Theorem 2.2]{mauacc} $M$ has finitely many minimal prime in $M$ submodules.
\end{rem}

\begin{cor}\label{1.13}
Let $M$ be projective in $\sigma[M]$. If $M$ is co-semisimple and a Goldie module then the maximal fully invariant submodules of $M$ are the minimal prime in $M$ submodules. Moreover, if $P_1,...,P_n$ are the minimal prime in $M$ submodules then $Spec(M)=\{P_1,...,P_n\}$. 
\end{cor}

\begin{proof}
Let $P$ be a minimal prime in $M$. Then $M/P$ is a prime Goldie module by Corollary \ref{4} and it is co-semisimple. By Proposition \ref{1.12} $M/P$ is $FI$-simple. 

Let $N\leq_{FI}M$ such that $P<N$. By \cite[Lemma 17]{raggiprime} $N/P\leq_{FI} M/P$, then $N=M$. Thus $P$ is a maximal fully invariant submodule. 

Now, if $N<M$ is a maximal fully invariant submodule then by \cite[Lemma 17]{raggiprime} $M/N$ is $FI$-simple. Since $M/N$ is $FI$-simple then $M/N$ is a prime module. By \cite[Proposition 18]{raggiprime} $N$ is prime in $M$.
\end{proof}

\begin{cor}\label{1.14}
Let $M$ be projective in $\sigma[M]$. If $M$ is co-semisimple and a Goldie module then for every fully invariant submodule $N\leq M$, the factor module $M/N$ is a semiprime Goldie module. 
\end{cor}

\begin{proof}
By \cite[Remark 4.23]{mauquantale} every fully invariant submodule $N$ of a co-semisimple module $M$ is semiprime in $M$. Let $N\leq M$. Since $N$ is semiprime in $M$, by \cite[Proposition 1.11]{maugoldie} 
\[N=\bigcap\{P\leq M|N\leq P\;\text{and}\;P\;\text{prime in}\;M\}\]

By Corollary \ref{1.13}, the prime in $M$ submodules are precisely the minimal ones. So $N$ is an intersection of minimal primes. By Proposition \ref{1.4} $M/N$ is a semiprime Goldie module.
\end{proof}

\begin{cor}\label{1.15}
Let $M$ be projective in $\sigma[M]$. If $M$ is co-semisimple and a Goldie module then $M$ is isomorphic to a finite product of $FI$-simple prime Goldie modules.
\end{cor}

\begin{proof}
Suppose that $P_1,...,P_n$ are the minimal prime in $M$ submodules. We have a monomorphism 
\[\varphi:M\to M/P_1\oplus...\oplus M/P_n\]
By Corollary \ref{1.13} each $P_i$ is a maximal fully invariant submodule, so $P_i+(\bigcap_{j\neq{i}}P_j)=M$. Therefore, $\varphi$ is onto. Thus $\varphi$ is an isomorphism.
\end{proof}

\begin{prop}\label{1.16}
Let $M$ progenerator in $\sigma[M]$. Suppose $M$ is co-semisimple and a Goldie module. If $E$ is an indecomposable injective module in $\sigma[M]$ such that $E$ embeds in $\widehat{M}$ then $E$ is $FI$-simple.
\end{prop}

\begin{proof}
Suppose that $P_1,...,P_n$ are the minimal prime in $M$ submodules. 
By Corollary \ref{4} and Proposition \ref{51} we have that 
\[\widehat{M}\cong \E(M/P_i)\oplus...\oplus \E(M/P_n)\]
with $\E(M/P_i)\cong E_i^{k_i}$ where $E_i$ is an indecomposable injective module in $\sigma[M]$ and $k_i>0$. By Proposition \ref{1.12} $M/P_i$ is $FI$-simple, so every fully invariant submodule of $\E(M/P_i)$ contains $M/P_i$.

Without lost of generality we can suppose $E^k=\E(M/P)$ for some $P\in\{P_1,...,P_n\}$. If $N\leq E$ is a fully invariant submodule then $N^k$ is a fully invariant submodule of $E^k$. Thus $M/P\leq_e N^k\leq_eE^k$. This implies that $\E(N^k)=\E(M/P)$. Therefore $N^k$ is a fully invariant submodule of its $M$-injective hull, so $N^k$ is quasi-injective.

Since $N$ is non $M$-singular, by Theorem \ref{1.9} $N^k$ is injective. Hence $N^k=E^k$ in particular $N$ is injective, thus $N=E$.
\end{proof}

\begin{cor}\label{1.17}
Let $M$ progenerator in $\sigma[M]$. Suppose $M$ is co-semisimple and a Goldie module, then every non $M$-singular indecomposable injective module is $FI$-simple.
\end{cor}

\begin{proof}
By \cite[Corollary 2.12]{mauacc} there exists a bijective correspondence between minimal prime in $M$ submodules and non $M$-singular indecomposable injective modules in $\sigma[M]$, up to isomorphism. Suppose $P_1,...,P_n$ are the minimal prime in $M$ submodules and let $E_1,...,E_n$ be the representatives of classes of isomorphism of non $M$-singular indecomposable injective modules in $\sigma[M]$. Now, by \cite[Theorem 2.20]{mauacc} $\widehat{M}\cong E_1^{k_1}\oplus\cdots E_n^{k_n}$. 

Let $E\in \sigma[M]$ be a non $M$-singular indecomposable injective module. Then $E\cong E_i$ for some $1\leq i\leq n$, so $E$ embeds in $\widehat{M}$. Thus, by Proposition \ref{1.16} $E$ is $FI$-simple.
\end{proof}

\begin{cor}\label{1.18}
Let $M$ progenerator in $\sigma[M]$. Suppose $M$ is co-semisimple and a prime Goldie module. If there exists a non $M$-singular indecomposable injective module $E\in \sigma[M]$ such that $E$ is a duo module then $M$ is semisimple.
\end{cor}

\begin{proof}
We just have to notice that a $FI$-simple duo module is simple. Now, if $M$ is a prime module then $\widehat{M}\cong E^k$ with $E$ a non $M$-singular indecomposable injective module. Let $Q$ be a non $M$-singular indecomposable injective duo module, by Corollary \ref{1.17} $Q\cong E$ so $E$ is simple. Thus $\widehat{M}$ is semisimple, then $M$ is semisimple.
\end{proof}

\bibliographystyle{plain}
\bibliography{biblio}

\end{document}